  \documentclass[12pt]{amsart}

\usepackage{amscd}
\usepackage{amsmath}
\usepackage{amssymb}
\usepackage{amsthm}
\usepackage{epsf}
\usepackage{latexsym}
\usepackage{verbatim}
\usepackage[pdftex]{graphicx}
\usepackage{tikz}
\input epsf.tex

\newtheorem{theorem}{Theorem}[section]
\newtheorem{lemma}[theorem]{Lemma}
\newtheorem{corollary}[theorem]{Corollary}
\newtheorem{remark}[theorem]{Remark}
\newtheorem{proposition}{Proposition}[section]

\newcommand{\Nm}{\mathbb{N}}

\newcommand{\Rm}{\mathbb{R}}

\newcommand{\Zm}{\mathbb{Z}}

\numberwithin{equation}{section}

\begin{document} 

\title[Subsequence Rational Ergodicity of Rank-Ones]{Subsequence Rational Ergodicity of Rank-One Transformations}

\author[Bozgan]{Francisc Bozgan}
\address[Francisc Bozgan]{University of California, Los Angeles, CA 90095, US }
\email {Francisc Bozgan <franciscbozgan@gmail.com>}

\author[Sanchez]{Anthony Sanchez}
\address[Anthony Sanchez]{Arizona State University \\
Tempe, AZ 85287, USA}
\email{Anthony Sanchez <anthony.sanchez.1@asu.edu>}

\author[Silva]{Cesar E. Silva}
\address[Cesar E. Silva]{Department of Mathematics\\
     Williams College, Williamstown, MA 01267, USA}
\email{Cesar E. Silva <csilva@williams.edu>}

\author[Stevens]{David Stevens}
\address[David Stevens]{Williams College, Williamstown, MA 01267, USA}
\email{David Stevens <David.F.Stevens@williams.edu>}

\author[Wang]{Jane Wang}
\address[Jane Wang] {Princeton University \\
Princeton, NJ 08544, USA}
\email{Jane Wang <jywang@princeton.edu>}

\subjclass[2000]{Primary 37A40; Secondary
37A05} 
\keywords{Infinite measure-preserving, ergodic, rationally ergodic, rank-one}

\maketitle 

\begin{abstract}  We show that all rank-one transformations are subsequence boundedly rationally ergodic and  that there exist rank-one transformations that are not weakly rationally ergodic. 
\end{abstract}


\section{Preliminaries}

We consider standard Borel measure spaces, denoted $(X,\mathcal{B},\mu)$, where $\mu$ is   a nonatomic $\sigma$-finite measure. We are interested in the case when $\mu$ is infinite. We study invertible measure-preserving transformations $T: X \rightarrow X$, i.e., $T$ is invertible mod $\mu$, $T^{-1}A$ is measurable if and only if $A$ is, and    $\mu(A)=\mu(T^{-1}A)$ for all $A\in\mathcal B$. A set $A$ is invariant if $T^{-1}A=A $ (all our equalities are mod $\mu$). A transformation $T$ is {\bf ergodic} if for every  invariant set $A$, $\mu(A)=0$ or  $\mu(X\setminus A)=0$. A transformation $T$ is {\bf conservative} if for every measurable set $A$ of positive measure, there exists   $n\in\mathbb N$ such that $\mu(A \cap T^{-n}A)>0$. One can show  that $T$ is conservative and ergodic if and only if  every set $A$ of positive measure sweeps out: $\bigcup_{n=0}^\infty T^{-n}A=X$. 

It is well known that when $T$ is ergodic and finite measure-preserving, the ergodic theorem gives a quantitative estimate for the average number of visits to a measurable set  for almost every point; for example, 
the law of large numbers can be obtained as a consequence of the ergodic theorem. In infinite measure, however, the averages given by the ergodic theorem (of visits to a finite measure set) converge to $0$, and Aaronson has shown that there is no normalizing sequence of  constants for which the ergodic averages (using this sequence) of visits to a set converge to the measure of the set, see \cite[2.4.1]{Aa97}. In \cite{Aa77}, Aaronson defined the notions of weak rational ergodicity and rational ergodicity as giving quantitative estimates for ergodic averages for sets satisfying certain conditions, and extended them to bounded rational ergodicity in \cite{Aa79}. More recently, in \cite{Aa12}, he defined a notion called rational weak mixing that is stronger than weak rational ergodicity.

This paper is organized in three parts.  In the first part we prove that rank-one infinite measure-preserving transformations are subsequence boundedly rationally ergodic, a notion where the bounded rational ergodicity condition is satisfied only for a subsequence. A consequence is that all rank-one transformations are not squashable, thus Maharam transformations are not rank-one.  In the course of doing this we prove that rank-one transformations with a bounded sequence of cuts satisfy the stronger property of being boundedly rationally ergodic.   Our work builds  on the paper by   Dai,
Garcia, P\u{a}durariu and Silva \cite{DGPS}, where   these properties are proved for a large class of  rank-one transformations satisfying a condition  called exponential growth. 
We extend the techniques of \cite{DGPS} and remove these assumptions, so that we obtain the corresponding results for all
rank-one transformations and for all rank-one transformations with bounded cuts. 
For the proofs we proceed by first showing the perhaps more intuitive notions of weak rational ergodicity and subsequence weak rational ergodicity, and then extend them to the more general case.
Recently, after most of our work was completed, we learned of the paper of Aaronson, Kosloff and Weiss
where they prove that rank-one transformations with a bounded sequence of cuts satisfy a property
that implies bounded rational ergodicty \cite{AKW13}.

We also prove that there exist rank-one transformations that are not weakly rationally ergodic,
so not rationally ergodic (the first examples of non-rationally ergodic transformations were
Maharam transformations \cite{Aa77}).
Rank-one transformations have been studied extensively in ergodic theory and are a source
of important examples and counterexamples (see, e.g. \cite{Fe97} and \cite{DS09} for the 
finite measure-preserving and infinite measure-preserving cases, respectively). It is well known that in the finite measure-preserving case, rank-one transformations are generic under the weak (also called coarse) topology on the group of invertible transformations on a standard space. While it is expected that this result is also true in infinite measure, we have not found a reference with a proof so we have included a proof of this fact in Section \ref{sec:gen}.   For the genericity result, rather than working with the cutting and stacking construction definition of rank-one transformations (used in examples and in e.g. \cite{DGPS}), we work with the abstract definition of rank-one. As  Aaronson \cite[\textsection 10]{Aa12} has shown that the class of weakly  rationally ergodic transformations is meager in the group of measure-preserving transformations, a consequence of our results is that there exist rank-one transformations that are subsequence boundedly rationally ergodic but not weakly rationally ergodic. Rank-one examples were not known earlier and while we know existence we do not know an explicit construction. 

In the second part (Section \ref{sec:ztmr}) we study the properties of zero type and multiple recurrence. 
Zero type is an interesting property defined by Hajian and Kakutani whose spectral definition in the case of finite measure turns out to be equivalent to mixing. It is known not to be generic.
As is well known, Furstenberg \cite{Fu81} proved that all finite measure-preserving transformations satisfy the strong property of multiple recurrence. While it has been known for some time that this is not the case in infinite measure (see e.g.  \cite{DS09} for a discussion of these results), it is of interest to investigate the relationship of this with other properties. We prove the independence of multiple recurrence with zero type in rank-one.

Finally, in Section \ref{sec:rwm} we study the more recent property of rational weak mixing
and give conditions for rank-one transformations not to be rational weak mixing.

\subsection{Acknowledgements}
This paper was based on research done by the Ergodic Theory group of the 2013 SMALL Undergraduate Research Project at Williams College. Support for this project was provided by the National Science Foundation REU Grant DMS - 0353634 and the Bronfman Science Center of Williams College.  We would like to thank Alexandre Danilenko and  Andr\'es del Junco for useful comments.

\subsection{Notions of Rational Ergodicity}

Given an invertible, measure-preserving transformation $T$ and a set $F \subset X$ of positive finite measure, define the \textbf{intrinsic weight sequence} of $F$ to be $$u_n(F) = \frac{\mu(F \cap T^kF)}{\mu(F)^2}.$$ For the sum of these weights up to $n$ write  $$a_n(F) = \sum_{k=0}^{n-1} u_k(F).$$ A set $F$ is said to \textbf{sweep out} if $\mu \left( X \backslash \bigcup_{i=0}^\infty T^{i} F\right) = 0$. 
By Maharam's theorem, if there is a sweep out set of finite measure the transformation is conservative. A transformation $T$ is said to be \textbf{weakly rationally ergodic} \cite {Aa77} if there exists a set $F \subset X$ of positive finite measure that sweeps out such that for all measurable $A,B \subset F$, 
\begin{equation}\label{eqn:wre} 
\frac{1}{a_n(F)} \sum_{k=0}^{n-1} \mu(A \cap T^k B) \to \mu(A)\mu(B)
\end{equation}
as $n \rightarrow \infty$. 
If $T$ is weakly rationally ergodic it is conservative ergodic; furthermore  any set $T^i(F)$ will work in place of $F$.
Additionally,  $T$ is \textbf{rationally weakly mixing} \cite{Aa12} if there exists a sweep out 
set 
  $F \subset X$ of finite measure such that for all measurable $A, B \subset F$, we have that 
\begin{equation}
\label{eqn:rwm}
\frac{1}{a_n(F)} \sum_{k=0}^{n-1} |\mu(A \cap T^k B) - \mu(A)\mu(B) u_k(F)| \rightarrow 0.
\end{equation}

If the limit \eqref{eqn:wre} only holds along a subsequence $\{n_i\} \subset \Nm$, then we say that the transformation $T$ is \textbf{subsequence weakly rationally ergodic} along $\{n_i\}$ and if \eqref{eqn:rwm} holds along a subsequence, then we say that $T$ is \textbf{subsequence rationally weak mixing} along $\{n_i\}$. Rational weak mixing along a sequence implies weak rational ergodicity 
along the same sequence \cite{Aa12}.


\subsection{Rank-One Transformations} Let $T$ be an invertible measure-preserving transformation
on $X$. 
A    \textbf{Rokhlin column}, or \textbf{column}, is a collection of pairwise disjoint measurable sets $B, T(B), \ldots, T^{h-1} (B)$. We call any single such set in the column a \textbf{level}, and  $h$  the height of the column. In a column we imagine the levels being stacked on top of each other, with $B$ at the bottom and $T^{h-1}(B)$ at the top.

An invertible measure-preserving transformation $T$ is said to be {\bf rank-one} if there exits a
sequence of Rokhlin columns $C_n=\{B_n,$ $\ldots,$ $T^{h_n-1}(B_n)\}$ such  that for any measurable set $A \subset X$ of finite measure and $\varepsilon > 0$, there exists an $N$ such that for all $n \geq N$, we have that $\mu(A \triangle B_n^\prime) < \varepsilon$ for some $B_n^\prime$ a union of levels of $C_n$. 
It follows that such a $T$ is conservative ergodic.

It can be shown that every rank-one transformation can be constructed so that the 
sequence of columns $\{C_n\}$ is refining in the sense that every level of $C_n$ 
is a union of levels in the previous column $C_{n-1}$ (this was shown in the finite measure-peserving case in  \cite[Lemma 9]{B71} and the same ideas work in the infinite case--Proposition \ref{prop:refining}). Thus, we may assume that $B_{n-1} = \bigcup_{i \in E} T^i (B_n)$ for some $E \subset \{0, 1, \ldots, h_n-1\}$.

\subsection{Rank-Ones as Cutting and Stacking}

When constructing rank-one examples it is useful to think of a process resembling cutting and stacking. Our first column $C_0$ consists of a single measurable set of positive finite measure. 

 In each step, given $C_n$, we cut the column into $r_n$ subcolumns, where $r_n \geq 2$. That is, we divide $B_n$, the base of column $C_n$, into $r_n$ sets of equal measure. If we label these sets as $B_{n,0}, B_{n,1}, \ldots, B_{n,r_n-1}$, then our first subcolumn would be $B_{n,0}, T(B_{n,0}), \ldots, T^{h_n-1}(B_{n,0})$ and our other $r_n-1$ subcolumns would be defined similarly. Above any subcolumn, we may add any number of new levels, called \textbf{spacers}, under the condition that these new levels are also pairwise disjoint. Then, $C_{n+1}$ is constructed by stacking each subcolumn with its associated spacers under the next subcolumn. If $A$ is a level below a level $B$ in $C_{n+1}$, then we must have that $T(A) = B$, $\mu(A) = \mu(B)$, and $T$ is invertible on $A$. Thus, $C_{n+1}$ will consist of $r_n$ copies of $C_n$ possibly separated by spacers. 

Since all rank-one transformations may be constructed in this way, we spend some time developing the notation that we will use to refer to these constructions throughout the rest of the paper. 

Suppose that $T$ is a rank-one transformation. Given a column $C_n$ of $T$, we let $h_n$ be the height of the column, $w_n$ be the width (the measure of each level), and $r_n$ be the number of subcolumns that $C_n$ is cut into. Given a level $J$ in $C_n$, we denote the height of $J$ in $C_n$ by $h(J)$. If we fix $J$ a level in $C_n$, and an $m \geq n$, we define the \textbf{descendants} of $J$ in $C_m$ to be the set of levels in $C_m$ whose disjoint union is $C_n$. We then define $D(J,m)$ as the set of the heights of these levels. 

To form $C_{n+1}$ from $C_n$, we first cut $C_n$ into $r_n$ subcolumns, which we denote by $C_n[0], C_n[1], \ldots, C_n[r_n-1]$. Then, before stacking, we add $s_{n,k}$ spacers above each $C_n[k]$, $0 \leq k \leq r_n-1$, where each $s_{n,k}$ is in $\Zm_{\geq 0}$. 

For full generality here, we should allow for the possibility of spacers beneath the first subcolumn $C_{n}[0]$. However, one can see that this   is not necessary as any spacers placed under the first subcolumn $C_{n}[0]$ can be added as spacers above other subcolumns in later columns $C_m$, $m \geq n$. Therefore, we do not need spacers below the first, or any, subcolumn.

Then define $h_{n,k} = h_n + s_{n,k}$ for each $k$. If we let $$H_n = \{0\} \cup \left\{\sum_{k=0}^{i} h_{n,k} : 0 \leq i < r_n -1\right\},$$ we have that, for $J$ in $C_j$ and $N\geq j$,
\[D(J,N) = h(J) + H_j \oplus H_{j+1} \oplus \cdots \oplus H_{N-1}.\]
With this notation we can easily find the number of elements in $D(J,N)$ to be $|D(J,N)| = |H_j| \cdot |H_{j+1}| \cdots |H_{N-1}| = r_j \cdot r_{j+1} \cdots r_{N-1}$. For a level $J$ in $C_n$ and $m \geq n$, we define the maximum height of its descendants in $C_m$ to be $M_m = \max\{D(J, m)\}$.

\section{Weak Rational Ergodicity}

In this section we prove that all rank-one transformations are subsequence weakly rationally ergodic, and that all rank-one transformations with a bounded number of cuts are weakly rationally ergodic. 

We begin with a technical lemma that is used to estimate the    $\mu(J \cap T^k J)$,
for any level $J$, using knowledge of the descendants of $J$. We first state the lemma proved 
by Dai et. al. 
in \cite{DGPS}, as it will also be needed in a later section. This lemma was proved under the assumption of normality. 
We remove this assumption in Lemma~\ref{lem:approx}, but replace the equalities in the conclusion of
Lemma~\ref{lem:normal} with 
  $\varepsilon > 0$ estimates.
  A rank-one transformation is said to be  {\bf normal} 
if $s_{n, r_n-1} > 0$ for infinitely many values of $n$, i.e.,  at least one spacer is added above the rightmost subcolumn infinitely many times. 

\begin{lemma}  \cite[Lemma 2.1]{DGPS}
\label{lem:normal}  
Let $T$ be a normal, rank-one transformation, $J$ be a level, and $n \in \Nm$. Then, for every $N$ sufficiently large, we have $$\mu(J \cap T^k J) = w_N \cdot |D(J, N) \cap (k + D(J,N))|$$ for all $0 \leq k < n$. Consequently, $$\sum_{k=0}^{n-1} \mu(J \cap T^k J) = w_N \cdot \sum_{k=0}^{n-1} |D(J, N) \cap (k + D(J,N))|.$$
\end{lemma}

We now prove the following without assuming normality.

\begin{lemma} 
\label{lem:approx}
Let $T$ be a rank-one transformation. Fix $\varepsilon > 0$, $J$ a level, and $n \in \Nm$. Then, we can find an $N \in \Nm$ such that for all $m \geq N$, we have $$\left|\mu(J \cap T^k J) - w_m \cdot |D(J,m) \cap (k + D(J,m))| \right| < \varepsilon$$ for all $0 \leq k < n.$ Furthermore, we can find an $M \in \Nm$ such that for all $m \geq M$, $$\left| \sum_{k=0}^{n-1} \mu(J \cap T^k J) - w_m \cdot \sum_{k=0}^{n-1} |D(J,m) \cap (k + D(J,m))|\right| < \varepsilon.$$
\end{lemma}
\begin{proof} For $J$ a level, and $n$ fixed, we can find an $N$ such that for all $m \geq N$, we have that $w_m \cdot n < \varepsilon$. For all but the top $k$ levels of $C_m,$ the image $T^k$ of each level is still a level in $C_m$. Thus, the image of at most $k<n$ descendants $J_i \in D(J,N)$ under the transformation $T^k$ are not levels in $C_m$. But then, if $\hat J$ is the union of the descendants of $J$ in $C_m$ not in the top $k$ levels of $C_m$, we have that 
\begin{align*} 
\big|\mu(J & \cap T^k J)  - w_m \cdot |D(J,m) \cap (k + D(J,m))| \big| \\ & \leq  \big|\mu(\hat J \cap T^k \hat J) - w_m \cdot |D(J,m) \cap (k + D(J,m))| \big| + \mu(J \backslash \hat J) \\
& \leq  0 + w_m \cdot k  \\
& <  \varepsilon.
\end{align*} 
By choosing $M = \max \{N_k\}_{k=0}^{n-1}$ where each $N_k$ is such that $$\left|\mu(J \cap T^k J) - w_m \cdot |D(J,m) \cap (k + D(J,m))| \right| < \varepsilon/2^k$$ for all $m \geq M$, we have that the second conclusion holds as well. 
\end{proof}

The first step in our proof is to  show that all rank-one transformations satisfy condition \eqref{eqn:wre} on finite unions of levels when $F = I$, the level in $C_0$. For this we use the following lemma, whose proof is contained in the proof of Theorem 2.2 in  \cite{DGPS}.

\begin{lemma}  \cite[2.2]{DGPS}
\label{lem:pn}
Fix $n \in \Nm$, $N$ sufficiently large as a function of $n$, and $J$ the bottom level of some column $C_j$, and define $P(n) = |U|^2 \sum_{k=0}^{n-1} |V \cap (k + V)|$ where $U = H_0 \oplus H_1 \oplus \cdots \oplus H_{j-1}$ and $V = D(J,N) = H_j \oplus H_{j+1} \oplus \cdots \oplus H_{N-1}$. Then, if $M$ is the maximum value of $U - U$, we have that $$\sum_{k=M}^{n-1-M}|(U \oplus V) \cap (k + U \oplus V)| \leq P(n) \leq \sum_{k=-M}^{n-1+M}|(U \oplus V) \cap (k + U \oplus V)|.$$ 
\end{lemma}

The following theorem lifts the restriction of normality in  
 \cite[Theorem 2.2]{DGPS}.

\begin{theorem}\label{thm:levels} Let $T$ be a rank-one transformation. Then $T$ satisfies condition \eqref{eqn:wre} for $A$ and $B$, finite unions of levels, and $F = I$, the level in $C_0$. \end{theorem}

\begin{proof} Let $F$ be the level in $C_0$. We first show that \eqref{eqn:wre} holds for $J$ the bottom level of any $C_j$. Let $U$ and $V$ be as in the statement of Lemma \ref{lem:pn}. We notice then that $D(I,n) = U \oplus V$ and $\mu(J) = \mu(I)/|D(I,j)| = \mu(I)/|U|$. Then, if we could show that 
\begin{equation}
\label{eqn:limit}
\frac{w_N}{a_n(I)} \left(|U|^2 \sum_{k=0}^{n-1} |V \cap (k + V)|\right) \rightarrow \mu(I)^2,
\end{equation}
where $N$ is some function of $n$, we would have that for any $\varepsilon$, we could find $M$ large enough such that for all $n \geq M$, $$\left|\frac{w_N}{a_n(I)} \left(\sum_{k=0}^{n-1} |V \cap (k + V)|\right)  - \left(\frac{\mu(I)}{|U|}\right)^2\right| < \varepsilon/2.$$ We also have by Lemma \ref{lem:approx} that we could in addition choose the $N$'s so that 
$$\left| \frac{1}{a_n(I)}\sum_{k=0}^{n-1} \mu(J \cap T^k J) - \frac{w_N}{a_n(I)} \cdot \sum_{k=0}^{n-1} |D(J,N) \cap (k + D(J,N))|\right| < \varepsilon/2,$$ so $$\left| \frac{1}{a_n(I)} \sum_{k=0}^{n-1} \mu(J \cap T^k J) - \mu(J)^2\right| < \varepsilon$$ and we would have convergence, as desired. 

Now we show the convergence in \eqref{eqn:limit}. By Lemma \ref{lem:pn}, if $M$ is the maximum value of $U-U$ and $P(n) = |U|^2 \sum_{k=0}^{n-1} |V \cup (k + V)|$, then $$\sum_{k=M}^{n-1-M}|(U \oplus V) \cap (k + U \oplus V)| \leq P(n) \leq \sum_{k=-M}^{n-1+M}|(U \oplus V) \cap (k + U \oplus V)|.$$ 

As $U \oplus V = D(I,N)$, we have by similar reasoning as in the proof of Lemma \ref{lem:approx} that
\begin{align*}
\sum_{k=M}^{n-1-M} \left(\mu(I \cap T^k I) - |k| \cdot w_N\right) &\leq w_N \cdot P(n)\\
 \leq \sum_{k=-M}^{n-1+M} \left(\mu(I \cap T^k I) + |k| \cdot w_N\right).
 \end{align*}
Now, since for each $n$, we may choose $N$ such that $\sum_{k=M}^{n-1-M} |k| \cdot w_N$ and $\sum_{k=-M}^{n-1+M} |k| \cdot w_N$ are uniformly bounded by some $C$, we have that 
$$ \frac{w_N \cdot P(n)}{a_n(I)}  \geq \frac{1}{a_n(I)} \left(\sum_{k=M}^{n-1-M} \mu(I \cap T^k I) - C\right) $$
and 
$$\frac{w_N \cdot P(n)}{a_n(I)} \leq \frac{1}{a_n(I)}\left( \sum_{k=-M}^{n-1+M} \mu (I \cap T^k I) + C\right).$$

 But since the $\mu(I \cap T^kI)$ terms are bounded above by $1$ but have a divergent sum, the right hand sides of both of the above inequalities tend to $\mu(I)^2$, giving us \eqref{eqn:limit} and therefore \eqref{eqn:wre} for $A = B = J$. 

Now when $A$ and $B$ are two levels of the same column, we can argue as in the proof of Theorem 2.2 in \cite{DGPS}, therefore obtaining the result for different levels and therefore also for finite unions of levels. 
\end{proof}

Now that we have that the weak rational ergodicity condition holds on levels in $I$, we want to be able to say that it holds for all measurable sets $A, B \subset I$. This will be true if we assume that the rank-one transformation has a bounded number of cuts. To show this, we begin with the following lemma that will allow us to bound how many times the intervals $\{T^k I\}$ cover $I$ for $|k| \leq M_m$. 

\begin{lemma} 
\label{lem:counting}
Let $T$ be a rank-one transformation and $I$ the level in $C_0$. Then, the sets $\{T^k I\}$ for $|k| \leq M_m$ cover almost every point of $I$ between $|D(I,m)|$ and $2 \cdot |D(I, m)|$ times. 
\end{lemma}

\begin{proof} We first show that for almost every $x \in I$, 
\begin{equation}
\label{eqn:12bound}
|D(I,m)| \leq |\{k \in [-M_m, M_m] : \exists y \in I, T^k y =x\}| \leq 2 |D(I,m)|.
\end{equation}
We fix an $x \in I$ and let $J$ be the level in $C_m$ that contains $x$ and $J^\prime$ be any descendant of $I$ in $C_m$. Set $d = h(J) - h(J^\prime)$. Clearly, $-M_m \leq d \leq M_m$ and so $T^dJ^\prime = J$. This holds for all $J^\prime \in D(I,m)$. By construction, $T^d$ is a bijection between $J$ and $J^\prime$, so for all $x \in I$ we have $|D(I,m)| \leq |\{k \in [-M_m, M_m] : \exists y \in I, T^k y =x\}|$. 

To show the upper bound, consider $J, J^\prime \in D(I,m)$. We pick $x$ to be a non-endpoint of $J$. We may assume that $J$ is above $J^\prime$. Let $C_N$ be the first column such that the copy of $C_m$ in $C_N$ containing $x$ is not the bottom or top copy of $C_m$ in $C_N$. That is, the copy of $C_m$ containing $x$ is $C_{m,N}[n]$ and there are copies $C_{m,N}[n-1]$ below and $C_{m,N}[n+1]$ above it. 

However, we see that any descendent of $J^\prime$ in $C_{m,N}[\ell]$ for $\ell \geq n+2$ is at least $2h_m - (h(J) - h(J^\prime)) > h_m$ levels from $x$ in $C_N$ and any descendent of $J^\prime$ in $C_{m,N}[\ell]$ for $\ell \leq n-1$ is at least $h_m + (h(J) - h(J^\prime)) > h_m$ levels from $x$ in $C_N$. Hence, as $M_m \leq h_m$, only the descendants of $J^\prime$ in $C_{m,N}[n]$ and $C_{m,N}[n+1]$ can cover $x$ for $|k| \leq M_m$. Thus, as each non-endpoint $x$ in $J$ can be covered at most $2$ times by any descendant of $J^\prime$ and there are $|D(I, m)|$, for a.e $x \in I$ we have that \eqref{eqn:12bound} as desired. 

Hence, the sets $\{T^k I\}$ for $|k| \leq M_m$ cover almost every point of $I$ at least $|D(I,m)|$ times and at most $2 \cdot |D(I, m)|$ times. 
 \end{proof}

We now show that  for transformations with a bounded number of cuts,  we can extend the weak rational ergodicity condition to all sets. Since by Theorem \ref{thm:levels}, we have \eqref{eqn:wre} for finite unions of levels, to prove it for arbitrary measurable sets, we will need some approximation results. The first of these is contained in the following lemma, whose proof we omit as it can be found in the proof of Theorem 2.3 in \cite{DGPS}. We remark here that our statement of this lemma is slightly different to that of \cite{DGPS} as our $I$ might not have measure one. 

\begin{lemma} 
\label{lem:mapprox}
Let $I$ be the level in $C_0$, $A, B \subset I$, and fix an $n \geq 0$. Then, choose $m$ such that $M_{m-1} \leq n < M_m$, where $M_m = \max\{D(I,m)\}$. If there exists a constant $c$ such that 
\begin{equation} 
\label{eqn:bound2} 
\sum_{k=0}^{M_m} \mu(I \cap T^k B_1) \leq c \mu(B_1) \sum_{k=0}^{M_{m-1}} \mu(I \cap T^k I). 
\end{equation} 
holds for all $B_1 \subset I$, then
\begin{equation} 
\label{eqn:bound1} 
\frac{1}{a_n(I)} \sum_{k=0}^{n-1} \mu(I \cap T^k B_1) \leq c \mu(B_1) \mu(I)^2, 
\end{equation} 
and therefore 
\begin{equation}
\label{eqn:approx} 
\left|\frac{1}{a_n(I)} \sum_{k=0}^{n-1} \mu(A \cap T^kB) - \frac{1}{a_n(I)} \sum_{k=0}^{n-1} \mu(A \cap T^k D)\right| \leq c\varepsilon, 
\end{equation} 
 where $D$ is a finite union of levels for which $\mu(B \triangle D) < \epsilon$.
\end{lemma}


With this lemma, we  prove the following theorem about the weak rational ergodicity of rank-one transformations with a bounded number of cuts; this was proved in \cite[2.3]{DGPS}
under the assumption of exponential growth, a condition that implies normality. 

\begin{theorem} 
\label{thm:wre}
Let $T$ be a rank-one transformation with a bounded number of cuts. Then $T$ is weakly rationally ergodic. 
\end{theorem}

\begin{proof}  By Theorem \ref{thm:levels}, we have weak rational ergodicity for finite unions of levels. To prove the result for arbitrary measurable sets, we establish some approximation estimates. The first estimates are as in the proof of Theorem 2.3 in \cite{DGPS}. 

We let $A,B \subset I$ be measurable sets. Given $\varepsilon > 0$, let $D$ be a finite union of levels so that $\mu(B \Delta D) < \varepsilon$. We need to show that there is a constant $c$ such that \eqref{eqn:approx} holds for all $n$. By Lemma \ref{lem:mapprox}, it suffices to show that there exists a $c$ such that the bound in \eqref{eqn:bound2} holds for all $B_1 \subset I$.

By Lemma \ref{lem:counting}, the sets $\{T^k I\}$ for $-M_m \leq k \leq M_m$ cover almost every point of $I$ at most $2\cdot |D(I,m)|$ and at least $|D(I,m)|$ times. Since $B_1 \subset I$, we have 
\begin{equation} 
\label{eqn:capstone1} 
\sum_{k = -M_m}^{M_m} \mu(I \cap T^k B_1) = \sum_{k = -M_m}^{M_m} \mu(T^kI \cap B_1) \leq 2 \mu(B_1) \cdot |D(I,m)| 
\end{equation} 
and 
\begin{equation} 
\label{eqn:multbyone1}
\sum_{k=-M_{m-1}}^{M_{m-1}} \mu(I \cap T^k I) \geq |D(I,m-1)| \mu(I).
\end{equation} 
From \eqref{eqn:capstone1} and \eqref{eqn:multbyone1}, we get $$ \sum_{k=-M_m}^{M_m} \mu(I \cap T^k B_1) \leq 2\left(\frac{|D(I,m)|}{|D(I, m-1)|}\right)\left(\frac{\mu(B_1)}{\mu(I)}\right)\left( \sum_{k=-M_{m-1}}^{M_{m-1}} \mu(I \cap T^k I)\right). $$ 
So \begin{align*}\sum_{k = 0}^{M_m} \mu(I \cap T^k B_1) &\leq 2 \frac{\mu(B_1)}{\mu(I)}\left(\frac{|D(I,m)|}{|D(I,m-1)|}\right)\left(2\left(\sum_{k=0}^{M_{m-1}} \mu(I \cap T^k I)\right) - 1\right) \\ &\leq 4\frac{\mu(B_1)}{\mu(I)}\left(\frac{|D(I,m)|}{|D(I,m-1)|}\right)\sum_{k=0}^{M_{m-1}} \mu(I \cap T^k I).\end{align*} 
Since $T$ has bounded cuts, we have that $|D(I,m)|/|D(I,m-1)| = r_{m-1} \leq K$ for all $m$, for some $K < \infty$.  Upon setting $c = 4K / \mu(I)$, we have satisfied \eqref{eqn:bound2}. 

Therefore \eqref{eqn:approx} holds, showing that we can approximate a measurable set $B$ with a finite unions of levels. Applying a similar argument to $A$, we can approximate $A$ with a finite union of levels as well, showing that it suffices to prove \eqref{eqn:wre} for finite unions of levels. But we showed this in Theorem \ref{thm:levels} so we are done. 
\end{proof}

Since Theorem \ref{thm:levels} gives us a form of weak rational ergodicity on levels for rank-one transformations, one could hope that all rank-one transformations are weakly rationally ergodic. In Section \ref{sec:gen}, we will see that this is not true as we will use genericity results to show that there exist rank-one transformations that are not weakly rationally ergodic. However, we do have that all rank-one transformations are subsequence weakly rationally ergodic, which is the next result we show;
this is in \cite[2.4]{DGPS} under the assumption of exponential growth.

\begin{theorem} 
\label{thm:swre}
All rank-one transformations are subsequence weakly rationally ergodic. 
\end{theorem}

\begin{proof}  
Given a rank-one transformation, we claim that it is subsequence weakly rationally ergodic along the subsequence $\{M_m + 1\}$. Following the proof of the previous theorem, Theorem \ref{thm:wre}, to prove subsequence weak rational ergodicity, it suffices to show that for all $B_1 \subset I$, 
\begin{equation}
\label{eqn:AB1bound}
\frac{1}{a_{n_i}(I)}\sum_{k=0}^{{n_i}-1} \mu(A \cap T^k B_1) \leq c \mu(B_1)
\end{equation}
is true along the subsequence $n_m = M_m+1$ are true. But we have that 

$$\sum_{k=0}^{n_m-1} \mu(A \cap T^k B_1) \leq \sum_{k=0}^{n_m-1} \mu(I \cap T^k B_1) = \sum_{k=0}^{M_m} \mu(I \cap T^k B_1)$$ 
and $$ \sum_{k=0}^{M_m} \mu(I \cap T^k I)  = \sum_{k=0}^{n_m-1} \mu(I \cap T^kI) = a_{{n_m}}(I) \cdot \mu(I)^2.$$ 
Thus, to show \eqref{eqn:AB1bound} it suffices to show that there is a $c$ for which 
\begin{equation}
\label{eqn:sumbound}
\sum_{k=0}^{M_m} \mu(I \cap T^k B_1) \leq c \mu(B_1) \sum_{k=0}^{M_m} \mu(I \cap T^k I). 
\end{equation} 
Now, as in the proof of the previous theorem, since the sets $\{T^k I\}$ for $-M_m \leq k \leq M_m$ cover almost every point of $I$ at most $2 \cdot |D(I,m)|$ times and at least $|D(I,m)|$ times, we get that 
$$ \sum_{k=-M_m}^{M_m} \mu(I \cap T^k B_1) \leq 2\left(\frac{|D(I,m)|}{|D(I, m)|}\right) \left(\frac{\mu(B_1)}{\mu(I)}\right)\left( \sum_{k=-M_m}^{M_m} \mu(I \cap T^k I)\right). $$ 
So \begin{align*}\sum_{k = 0}^{M_m} \mu(I \cap T^k B_1) &\leq 2 \left(\frac{|D(I,m)|}{|D(I,m)|}\right)\left(\frac{\mu(B_1)}{\mu(I)}\right)\left(2\left(\sum_{k=0}^{M_m} \mu(I \cap T^k I)\right) - 1\right) \\ &\leq 4\frac{\mu(B_1)}{\mu(I)}\sum_{k=0}^{M_m} \mu(I \cap T^k I).\end{align*} 
Setting $c = 4/ \mu(I)$, we satisfy \eqref{eqn:sumbound}. Therefore \eqref{eqn:AB1bound} holds, showing that we can approximate a measurable set $B$ with a finite unions of levels along the subsequence $\{M_m +1\}$. We can prove a similar approximation statement for $A$. Now by Theorem \ref{thm:levels}, all rank-one transformations are subsequence weakly rationally ergodic along the subsequence $\{M_m + 1\}$. \end{proof}

\section{The Centralizer for Rank-One Transformations}

Recall that $S:X\to X$ is  an \textbf{invertible nonsingular} transformation if it is invertible, measurable and  $\mu(A)=0$ if and only if $\mu(S(A))=0$.   The (nonsingular) \textbf{centralizer} $C(T)$ of a transformation $T$  consists of all invertible nonsingular $S$ such that $ST = TS$. It was shown by Aaronson \cite{Aa77} that if an invertible $S$ is in the  nonsingular centralizer of a weakly rationally ergodic $T$, then $S$ is measure-preserving. In fact, this result holds even if $T$ is only subsequence weakly rationally ergodic, with essentially the same argument,
which we include below for completeness. As a consequence, we have that any invertible, nonsingular $S$ in the centralizer of a rank-one transformation is measure-preserving.  Recall that $T$ is called 
\textbf{squashable} if it commutes with a non-measure-preserving $S$. Regarding the centralizer we note that Ryzhikov \cite{Ry11} has shown 
that the centralizer of a zero type conservative ergodic infinite measure-preserving transformation 
is trivial, i.e., consists just of powers of the transformation. In the finite measure-preserving case, it is
well known that for rank-one transformations an element $S$ of the centralizer is a limit in the weak topology of powers of $T$ \cite{Ki86}.

\begin{proposition} Let $T$ be a conservative ergodic measure-preserving transformation. If $T$ is subsequence weakly rationally ergodic and $S$ is an invertible, nonsingular transformation in $C(T)$, then $S$ is measure-preserving, thus $T$ is not squashable.
\end{proposition}

\begin{proof} It is well-known that if $S$ is nonsingular and commutes with $T$, then there exists a $c$ such that $\mu (S (A)) = c \cdot \mu(A)$ for all measurable $A \subset X$. Since $T$ is subsequence weakly rationally ergodic, there exists a set $F \subset X$ that sweeps out and a sequence $\{n_i\} \subset \Nm$ on which for all $A,B \subset F$, 
$$ \lim_{i \rightarrow \infty} \frac{1}{a_{n_i}(F)} \sum_{k=0}^{n_i} \mu (A \cap T^k B) = \mu(A)\mu(B).$$ By the invertibility of $S$,   for any $C, D \subset SF$, there exist $A, B \subset F$ such that $C = S A, D = SB$. Then, as $\mu \circ S^{-1} = c^{-1} \mu$, and $S$ and $T$ commute,  
$$\lim_{i \rightarrow \infty} \frac{1}{a_{n_i}(SF)} \sum_{k=0}^{n_i} \mu(C \cap D) = \mu(C) \mu (D)$$ for all $C, D \subset SF$. As the subsequence weak rational ergodicity property holds on $F$ and $SF$, it also holds on $F \cup SF$ \cite{Aa12},  so for $A \subset F$,  
\begin{align*}
\mu (SA) \mu (SA) & = \lim_{i \rightarrow \infty} \frac{1}{a_{n_i}(F \cup SF)} \sum_{k=0}^{n_i} \mu (SA \cap T^k SA) \\
& = c \cdot \lim_{i \rightarrow \infty} \frac{1}{a_{n_i}(F \cup SF)} \sum_{k=0}^{n_i} \mu (A \cap T^k A) \\
& = c \mu (A) \mu (A).
\end{align*}
But we also have that $\mu(SA)^2 = c^2 \mu(A)^2$, so   $c = 1$.
\end{proof}

As a result of this proposition and Theorem \ref{thm:swre}, we have the following corollary. 

\begin{corollary}\label{C:centmp} It $T$ is a rank-one transformation and $S$ is in the centralizer of $T$, then $S$ is measure-preserving. Thus $T$ is not squashable.
\end{corollary}

 Our result has an interesting consequence for Maharam transformations, which are a special class of infinite measure-preserving transformations arising from nonsingular transformations.  If $T$ is
 an invertible nonsingular transformation on $X$, we can define the Radon-Nikodym derivative $\omega(x)=\frac{d\mu\circ T}{d\mu}(x)$. Then define a space $X^*=X\times \mathbb R^+$ with  measure $\mu*=\mu\times\lambda$. The \textbf{Maharam transformation} $T^*$ is defined on $X^*$ 
 by $T^*(x,y)=(Tx,y/\omega(x))$.
It can be shown that $T^*$ is measure-preserving with respect to the infinite measure 
$\mu*$ and is conservative if and only if $T$ is \cite{Ma64}. There are cases when $T^*$ is also ergodic, namely when the
nonsingular transformation $T$ is type III (see e.g. \cite{DS09}). There are also conservative 
ergodic Maharam $\mathbb Z$-extensions (see e.g., \cite{Aa97}), and our result would also apply to them.

\begin{corollary}
Let $T$ be a conservative nonsingular transformation. The Maharam transformation $T^*$ is  not rank-one. 
\end{corollary}

\begin{proof} It is clear when $T^*$ is not ergodic. Then assume $T^*$ is ergodic. It is well known that $T^*$ commutes with
the non-measure-preserving  transformations $Q(x,y)=(x,a y)$ for any positive constant $a$.
Then use Corollary~\ref{C:centmp}.
\end{proof}

\section{Bounded Rational Ergodicity}

In this section, we extend the results of the previous section and show that Theorems \ref{thm:wre} and \ref{thm:swre} still hold if we replace the conclusions of weak rational ergodicity and subsequence weak rational ergodicity with bounded rational ergodicity and subsequence bounded rational ergodicity. 

Let  $f: X \rightarrow X$ is a measurable function. Then define $$S_n(f) := \sum_{k=0}^{n-1} f \circ T^k.$$ A transformation $T$ is said to be \textbf{rationally ergodic} if there exists an $M < \infty$ and $F \subset X$ that sweeps out of positive finite measure such that \textbf{Renyi's Inequality} is satisfied:
\begin{equation}
\label{eqn:re}
\int_F (S_n(1_F))^2 d \mu \leq M \left(\int_F S_n(1_F)d\mu\right)^2
\end{equation}
for all $n \in \Nm$. Furthermore, we say that $T$ is \textbf{boundedly rationally ergodic} if there exists an $F \subset X$ of positive finite measure that sweeps out such that 
\begin{equation}
\label{eqn:bre}
\sup_{n \geq 1} \left\| \frac{1}{a_n(F)} S_n (1_F) \right\|_\infty < \infty.
\end{equation}
If \eqref{eqn:re} holds only for a subset $\{n_i\} \subset \Nm$, then we say that $T$ is \textbf{subsequence rationally ergodic} and if \eqref{eqn:bre} holds only for a subset $\{n_i\} \subset \Nm$, then we say that $T$ is \textbf{subsequence boundedly rationally ergodic}. 

\begin{proposition}
\label{prop:supimp}
Let $T$ be a rank-one transformation and $I$ be the level in $C_0$. If for all $B \subset I$ and for a fixed $n \in \Nm$, we have that 
\begin{equation} 
\label{eqn:bound3} 
\frac{1}{a_n(I)} \sum_{k=0}^{n-1} \mu(I \cap T^k B) \leq c \mu(B), \end{equation} then for that $n$,
\begin{equation}
\label{eqn:brebound}
\left\| \frac{1}{a_n(I)} S_n (1_I) \right\|_\infty \leq c
\end{equation}
\end{proposition}

\begin{proof} We fix an $x$ in our space and show that 
\begin{equation}
\label{eqn:brebound1}
\frac{1}{a_n(I)} \sum_{k=0}^{n-1} 1_I \circ T^k(x) \leq c.
\end{equation}
We notice that it suffices to show \eqref{eqn:brebound1} holds for $x \in I$. For $x \not \in I$, we either have that $T^j (x) = y$ for some $y \in I$, $0 \leq j \leq n-1$ or there are no such $y$ and $j$. In the former case, we take the least such $j$ and have that $\frac{1}{a_n(I)} \sum_{k=0}^{n-1} 1_I \circ T^k(x) \leq \frac{1}{a_n(I)} \sum_{k=0}^{n-1} 1_I \circ T^k(y) \leq c$, and in the latter we have that $\frac{1}{a_n(I)} \sum_{k=0}^{n-1} 1_I \circ T^k(x) = 0.$

Now, for every $x \in I$, and a fixed $n \in \Nm$, we have that there exists a column $C_m$ for which if $J \ni x$ is a level in $C_m$, there are at least $n$ levels above $J$. Then, we see that $$\frac{1}{a_n(I)} \sum_{k=0}^{n-1} 1_I \circ T^k (x) = \frac{1}{a_n(I)} \sum_{k=0}^{n-1} \frac{\mu(I \cap T^k J)}{\mu(J)} \leq c.$$ Now, since \eqref{eqn:brebound1} holds for all $x \in I$ and therefore all $x \in X$, we have that \eqref{eqn:brebound} holds. 
\end{proof}

From this proposition and intermediate steps in our proofs of Theorems \ref{thm:wre} and \ref{thm:swre}, we can obtain a couple of simple corollaries, as Lemma \ref{lem:mapprox} gives us that \eqref{eqn:bound2} implies \eqref{eqn:bound3}.

\begin{corollary} All rank-one transformations with bounded cuts are boundedly rationally ergodic.
\end{corollary}

\begin{corollary} All rank-one transformations are subsequence boundedly rationally ergodic. 
\end{corollary}
\begin{proof} In the proof of Theorem \ref{thm:swre}, we showed that for any rank-one transformation $T$, we have that \eqref{eqn:bound3} holds for all $n \in \{M_m+1\}_{m \in \Nm}$ and $A, B \subset I$. Therefore, Proposition \ref{prop:supimp} gives us that \eqref{eqn:bre} holds for $n \in \{M_m + 1\}$, implying subsequence bounded rational ergodicity. 
\end{proof}

We notice that since bounded rational ergodicity implies rational ergodicity, both of the above results hold for rational ergodicity as well. 

\begin{remark}
As mentioned in the introduction, it was recently proven in \cite{AKW13} that all rank-one cutting and stacking transformations $T$ with a bounded number of cuts satisfy the property:
\begin{equation*}
\exists A, 0 < \mu (A) < \infty, \text{ for which }\sum_{|k|\leq n} 1_A(T^kx) \asymp a_n(A), \mbox{ a.e.},
\end{equation*}
where for positive sequences $a_n$ and $b_n$, $a_n \asymp b_n$ means the existence of $M>0$ such that $\frac{1}{M}<\frac{a_n}{b_n}<M$, $\forall n$ large. This property implies bounded rational ergodicity
\cite{AKW13}.  Moreover,  the above property for subsequences can be generalized for all rank-one transformations using the ideas of Propostion \ref{prop:supimp}, namely if $T$ is a rank-one transformation, then 
\begin{equation*}
\exists A, 0 < \mu (A) < \infty, \text{ for which } \sum_{|k|\leq n} 1_A(T^kx) \asymp a_n(A), \mbox{ a.e.}, \text{ for } n \in \kappa,
\end{equation*}
which implies subsequence bounded rationally ergodic along the subsequence $\kappa=\{M_m+1\}.$
\end{remark}

\section{Genericity of Rank-One in Infinite Measure}
\label{sec:gen}

\subsection{Some Definitions and Preliminaries}
We first discuss the \textbf{weak topology} on the space of all invertible, measure-preserving transformations on $X = \Rm_{\geq 0}$, with Lebesgue measure, which we call $\mathcal{M}$. As mentioned in the introduction, in the finite measure case this result is well known (see e.g. \cite{Ki00} and the references therein). The first discussion of genericity properties in infinite measure is probably  in \cite{Sa71}.  We say that a sequence of transformations $\{T_n\}$ converges to $T$ if and only if $\{T_n A\}$ converges to $TA$ for all $A \subset X$ such that $\mu(A) < \infty$. That is, $\mu(T_n A \triangle TA) + \mu(T_n^{-1} A \triangle T^{-1} A) \rightarrow 0$. 

If $\{A_i\}$ is a dense collection of sets for the finite measure sets in the Borel sigma  algebra $\mathcal{B}$ of $X$, then define a metric on $\mathcal M$ as follows: 
\[d(T, S) = \sum_{i = 1}^\infty \frac{1}{2^i} \left(\mu(TA_i \triangle SA_i) + \mu(T^{-1} A_i \triangle S^{-1} A_i)\right).\]
This metric generates the weak topology on the space. For completeness, we explicitly define $\{A_i\}$ here. We define the sets in stages. In the first stage, we let $A_1$ be $[0,1)$. In the second stage, we let the next sixteen sets, $A_2, \ldots, A_{17}$, be all combinations of unions of the dyadic intervals $[0,1/2)$, $[1/2, 1)$, $[1, 3/2)$, $[3/2, 2)$. Inductively, in the $n$th stage, we suppose that all sets in the previous $n-1$ stages have been defined and we define the next $2^{n \cdot 2^{n-1}}$ sets $A_i$ to be all unions of combinations of dyadic intervals of length $\frac{1}{2^{n-1}}$ in $[0, n)$. Any finite union of dyadic intervals is in $\{A_i\}$, and the $\{A_i\}$ are dense in the collection of finite measurable sets in $\mathcal{B}$. We can verify that under this choice of $\{A_i\}$, the metric always returns a finite distance. 

We say that a transformation $T$ is a \textbf{cyclic permutation of rank} $k$ if $T$ is the identity on $[k, \infty)$ and there exists $I_k$ a dyadic interval in $[0,k)$ of length $1 /2^{k-1}$ such that $I_k, TI_k, T^2I_k, \ldots, T^{k \cdot 2^k - 1}I_k$ is a partition of $[0, k)$ consisting of disjoint half-open dyadic intervals of length $1 /2^{k-1}$. We call the set of cyclic permutations of rank $k$, $O_k$. 

\subsection{Genericity of Rank-ones}
We prove the genericity of rank-one transformations in $\mathcal{M}$ with the weak topology. 
We say that a (Rohklin) column approximates a measurable set of finite measure $A$ with accuracy $\varepsilon > 0$ if there exists $B^\prime$, a union of levels, such that $\mu(A \triangle B^\prime) < \varepsilon$. A sequence of columns $\{C_n\}$ approximates $A$ if for any $\varepsilon > 0$ there exists some $N_\varepsilon \in \Nm$ such that for all $n \geq N_\varepsilon$, $C_n$ approximates $A$ with accuracy $\varepsilon$. So $T$ is rank-one if and only If there exists a sequence of columns $\{C_n\}$ such that $\{C_n\}$ approximates every finite measurable set $A \subset X$.

\begin{theorem} The rank-one transformations are generic in $\mathcal{M}$ under the weak topology.
\end{theorem}

\begin{proof} We  show that  the rank-one transformations contain a dense $G_\delta$ set in $\mathcal{M}$. To find this set,  define 
\begin{align*} 
\mathcal{R}_j := \{T \in \mathcal{M} : T & \text{ has a Rokhlin column that approximates } \\
& A_1, \ldots, A_j \text{ with accuracy } 1/j\}.
\end{align*}
We claim now that \[\mathcal{R} = \bigcap_{j=1}^\infty \mathcal{R}_j\] is our desired dense $G_\delta$ set. It suffices to show that each $T \in \mathcal{R}$ is rank-one and each $\mathcal{R}_j$ is open and dense. Then, since $\mathcal{M}$ is a Polish and therefore a Baire space (see e.g., \cite{Aa97}), $\mathcal{R}$ is also dense. (We note here that Danilenko has informed us that the proof in \cite{Da08s} 
that the finite measure-preserving rank-ones are $G_\delta$ can be adapted to the infinite measure-preserving case.)

Now let $T \in \mathcal{R}$; we show that it is rank-one. For each $j$, there exists a Rokhlin column $C_j$ for $T$, $B_j, TB_j, T^2 B_j, \ldots, T^{h_j-1}B_j$, that approximates $A_1,$ $ \ldots,$ $A_j$ with accuracy $1/j$. But then, given any finite measurable set $A \subset X$ and any $\varepsilon > 0$, we can find an $A_i$ such that $\mu(A \triangle A_i) < \varepsilon/2$. Now, choosing an $N > \max\{i, 2/\varepsilon\}$, we have that for all $j \geq N$, there exists some $B_j^\prime$ a finite union of levels of $C_j$ that approximates $A_i$ within $1/j < \varepsilon/2$ and therefore $A$ within $\varepsilon$. Hence, $T$ has a sequence of Rokhlin columns $\{C_j\}$ that approximate every finite measurable set within $\varepsilon$, showing that $T$ is rank-one. 

We next  fix an $\mathcal{R}_j$ and first show that it is dense. We know that the cyclic permutations $\bigcup_{k=1}^\infty O_k$ are dense (\cite{Sa71}), where $O_k$ is the set of cyclic permutations of rank $k$. Here, a cyclic permutation is a transformation $T$ that is the identity on $[k,\infty)$ and the sets $I_k, TI_k, \ldots, T^{k \cdot 2^k} I_k$ is a partition of $[0, k)$ into sets of measure $1 /2^k$. 

As $A_1, \ldots, A_j$ are finite unions of dyadic intervals, there exists a $K$ such that for all $k \geq K$, each $A_i, 1 \leq i \leq j$ may be written as a finite union of the dyadic intervals $$[0, 1 /2^k), [1 /2^k, 2 / 2^k), \ldots, [k - 1 /2^k, k).$$ But for any $T \in O_k$, there exists a Rokhlin column $I_k, TI_k, T^2I_k,$ $ \ldots, $ $T^{k \cdot 2^k - 1}I_k$ of exactly these dyadic intervals. Hence, each $A_i$ may be written as a union of the levels in this column, so $O_k \subset \mathcal{R}_j$ for all $k \geq K$. Since all but finitely many of the cyclic permutations are in any $\mathcal{R}_j$, $\mathcal{R}_j$ contains a dense set, so is dense. 

To   show  each $\mathcal{R}_j$ is open in $\mathcal{M}$ let $T \in \mathcal{R}_j$. Thus, there exists a Rokhlin column $C_j$ for $T$ that approximates each $A_1, \ldots, A_j$ within some $a < 1/j$. We let the levels be $B, TB, T^2B, T^{h-1}B$. The idea now is that we want to find a $\delta$ small enough such that if $S \in B(T, \delta)$, then $\mu(T(L) \triangle S(L))$ for each level $L$ of $C_j$ is small enough that we can find another sequence of levels $C, SC, S^2C, \ldots, S^{h-1}C$ where each $S^i C$ is close to $T^i B$. Therefore, the new Rokhlin column, which we call $C_j^\prime$, consisting of $C, SC, S^2C,$ $ \ldots,$ $ S^{h-1} C$ still approximates $A_1, \ldots, A_j$ within $1/j$. 

To prove this is an exercise in approximation. If we let $\varepsilon = 1/j - a$ and had that $\mu(S^i C \triangle T^iB) < \varepsilon / h$ for all $i$, then we could approximate $A_i$ within $a + \varepsilon = 1/j$ by $C_j^\prime$ as we may approximate any $A_i$ within $a$ by $C_j$. 

Now, if we had that 
\begin{equation}
\label{eqn:levelbound}
\mu(S(L) \triangle T(L)) < \frac{\varepsilon}{h}
\end{equation}
 for all levels $L$ of $C_i$ but the top one, then letting $C_0 = B, C_1 = SC_0 \cap TB, C_2 = SC_1 \cap T^2B, \ldots$, all the way up to $C_{h-1} = SC_{h-2} \cap T^{h-1} B$, we could define $S^{h-1} C = C_{h-1}$. Since $S$ is invertible, we could then define $S^i C = S^{i-h+1} C_{n-1}$ for all $0 \leq i \leq h-1$ and notice that $S^{i} C \subset T^iB$ for all such $i$. Furthermore, these levels would be disjoint and of the same measure. 
 
In particular, from \eqref{eqn:levelbound}, we know that the measure of $S^{h-1} C \subset T^{h-1}B$ is at least 
$$ \mu(B)-h\cdot \frac{\varepsilon}{h}= \mu(B) - \varepsilon,$$
as $\mu(SC_i \cap T^{i+1} B) = \mu(SC_i \cap T(T^iB)) \geq \mu(C_i) - \epsilon/h$ for all $i$. Then, as $S^i C \subset T^i B$ for all $i$ and $S$ is measure-preserving, we have that $\mu(S^i C \triangle T^i B) < \epsilon$ for $0 \leq i \leq h-1$. 

Finally, we need to show that we can obtain the bound in \eqref{eqn:levelbound}, that 
$$\mu(S(L) \triangle T(L))< \left(\frac{\varepsilon}{h}\right)$$
for all levels $L$ of $C_i$ and for all $S \in B(T, \delta)$ for small enough $\delta$. We fix a level $L$ of $C_i$. Then, we can approximate $L$ within $\varepsilon/4$ by a finite union of dyadic intervals $D_1, \ldots, D_n$. We note that each $D_i = A_{l_i}$ for some $l_i$. Furthermore,by choosing $\delta < \min_{1 \leq i \leq n} \{ \frac{\varepsilon}{2^{l_i+1}n}\}$, we have that for $S \in B(T, \delta)$, $$\mu(T(D_i) \triangle S(D_i)) < \frac{1}{2n} \varepsilon.$$ Hence, we have that 
$$\mu(T(L) \triangle S(L)) < \sum_{i=1}^n \mu(T(D_i) \triangle S(D_i)) + 2\mu\left(L \backslash \bigcup_{i=1}^n D_i\right) < \varepsilon$$ 
for $S \in B(T, \delta)$. For each level $L_i$ there exists such a $\delta_i$, and choosing $\delta = \min \{\delta_i\}$, the above inequality holds for each level $B, TB, \ldots, T^{h-1}B$ of $C_i$ for all $S \in B(T, \delta)$, and we have the bound \eqref{eqn:levelbound} as desired. Hence, $B(T, \delta) \subset \mathcal{R}_j$, showing that each $\mathcal{R}_j$ is open. 
\end{proof}

\subsection{Existence of a Rank-one Transformation that is not Weakly Rationally Ergodic}

Aaronson proved  that weak rational ergodicity is meager in the space of measure-preserving, invertible transformations under the weak topology (\cite{Aa12},  proof of Theorem F). But as  rank-one transformations are generic in the same space, we have the following corollary: 

\begin{corollary} There exist rank-one transformations that are not weakly rationally ergodic. 
\end{corollary}

With this, we have shown that all rank-ones are subsequence boundedly rationally ergodic, and weakly rationally ergodic on levels, but that there exist rank-one transformations that are not weakly rationally ergodic for all sets.

\section{Independence of the Zero Type Property and Multiple Recurrence in Rank-One}
 \label{sec:ztmr}

In this section, we show that the zero type property and multiple recurrence are independent. We start with some definitions. Recall  that $T$ is $k$-\textbf{recurrent} if for an $A \subset X, \mu(A) > 0$, there exists an $n \in \Nm$ such that $$\mu(A \cap T^{n}A \cap T^{2n}A \cap \ldots \cap T^{(k)n}A) > 0.$$ Finally, we say that $T$ is \textbf{multiply recurrent} if it is $k$-recurrent for all $k \in \Nm$. 

Furthermore, we say that a transformation $T$ on a space $X$ is of \textbf{zero type} if $\mu(A \cap T^{k} A) \rightarrow 0$ for all $A \subset X$ of finite measure. For conservative ergodic transformations, it suffices to check that this property holds for one set $A$ of finite positive measure. 

Finally, we note that our constructions will rely on the notion of steepness. We say that a rank-one transformation is \textbf{steep} if $t_{i+1} \geq 5 t_i$ for every pair of successive $t_i, t_{i+1}$ in $H = \bigcup_{j=0}^\infty H_j \backslash \{0\}$.

\subsection{Multiply Recurrent, Zero Type}

A non-rank-one example is given by an infinite ergodic index Markov shift, it is zero type 
and multiply recurrent by \cite{AN00}.
For a   rank-one example we begin with $C_0$, the unit interval. Then, for each $n \geq 0$, we let $r_n = n+2$, and add the appropriate number of spacers above each subcolumn to guarantee the following sets of heights. We never add spacers on the first subcolumn. For each $n$, we want 
\begin{align*} 
H_n = \{0\} & \bigcup \left\{5^{\frac{n(n+1)}{2} }, 2 \cdot 5^{\frac{n(n+1)}{2} }, \ldots,  \lceil \sqrt{n}\rceil  \cdot 5^{\frac{n(n+1)}{2} }\right\} \\ & \bigcup \left\{ 5^{\frac{n(n+1)}{2} + \lceil\sqrt{n}\rceil }, 5^{\frac{n(n+1)}{2} + \lceil\sqrt{n}\rceil + 1 }, \ldots 5^{\frac{n(n+1)}{2} + n}\right\}.
\end{align*}
We notice that in this construction, we place no spacers on the first $\lceil \sqrt{n} \rceil$ subcolumns in $C_n$ and have a form of steepness on the remainder of the subcolumns. 

We claim that this transformation is both zero type and multiply recurrent. Rank-one zero type transformations 
have been constructed in e.g. \cite{AdFrSi97}, \cite{GHPSW03},  \cite{DaRy11}, \cite{Ry11}, \cite{DGPS}. We also note that Danilenko has informed us that the high staircases (the tower staircases in \cite{BoFiMaSi01}) that are shown to be of zero type (mixing) in 
\cite{DaRy11} can be shown to be multiply recurrent.

\begin{proposition} $T$ as defined above is zero type. 
\end{proposition}
\begin{proof} We can easily see that $T$ is conservative and ergodic as it is a cutting and stacking transformation. Therefore, to show that it is zero type, it suffices to check that $\mu (A \cap T^n A) \rightarrow 0$ for one set $A$ of finite measure. We will prove this for $I = (0,1)$, the unit interval. 

By construction, $T$ is normal. Therefore, by Lemma \ref{lem:normal} we have that for large enough $N$, $$\mu(I \cap T^k I) = \frac{|D(I, N) \cap (k + D(I, N))|}{|D(I,N)|}.$$ We then have that $|D(I, N) \cap (k + D(I, N))|$ counts the number of representations 
\begin{equation}
\label{eqn: rep}
k = \sum_{i=0}^{N-1} (e_i - e_i^\prime), 
\end{equation}
where $e_i, e_i^\prime \in H_i$. We claim that if $$\sum_{i=0}^{N-1} (f_i - f_i^\prime) = \sum_{i=0}^{N-1} (e_i - e_i^\prime),$$ and $f_i, f_i^\prime, e_i, e_i^\prime \in H_i$, then $f_i - f_i^\prime = e_i - e_i^\prime$ for all $i$. 

To see this, we consider for each $n$ its difference set $D_n := \{e_n - e_n^\prime: e_n, e_n^\prime \in H_n\}$. We notice that for each $d_n \in D_n$, $d_n \equiv 0 \pmod{5^{\frac{n(n+1)}{2}}}$. Furthermore, for $d_n \in D_n$ and $d_{n+1} \neq 0 \in D_{n+1}$, we have that $|d_{n+1}| \geq 5 \cdot |d_n|$. This steepness of difference sets gives us unique representation by sums of differences. We notice that any $k$ has a unique representation as $\sum_{i=0}^{N-1} \varepsilon_i \cdot 5^i$ where $\varepsilon_i \in \{0, \pm 1, \pm 2\}$ for $N$ large enough. Then if $k$ has a representation \eqref{eqn: rep}, we must have $$\sum_{i=\frac{n(n+1)}{2} }^{\frac{(n+1)(n+2)}{2} - 1} \varepsilon_i \cdot 5^i = d_n,$$ where $d_n \in D_n$ and so the representation $k = \sum_{n=0}^{N-1} d_n$ is unique. Then, we have that $$|D(I, N)| = \prod_{i=0}^{N-1} |H_i|,$$ so for a fixed $k = \sum_{i=0}^{N-1} (d_i - d_i^\prime),$ we have that 
\begin{align*} 
\mu (I \cap T^k I) & = \prod_{i=0}^{N-1}\frac{|\{(e_i, e_i^\prime) \in H_i \times H_i : e_i - e_i^\prime = d_i - d_i^\prime\}|}{|H_i|}
\end{align*}

We see that if $d_i - d_i^\prime = 0$, then $|\{(e_i, e_i^\prime) \in H_i \times H_i : e_i - e_i^\prime = d_i - d_i^\prime\}| = |H_i|$. Else, we have that $|\{(e_i, e_i^\prime) \in H_i \times H_i : e_i - e_i^\prime = d_i - d_i^\prime\}| \leq \lceil \sqrt{i} \rceil.$ Now, for any $k \geq M_n$, we see that we must have $d_i - d_i^\prime \neq 0$ for at least one $i \geq n$. Hence, we have that $$\mu (I \cap T^k I) \leq \frac{\lceil \sqrt{i} \rceil}{i},$$ which goes to $0$ as $k \rightarrow \infty$. 
\end{proof}

We also claim that $T$ is multiply recurrent. Before we prove this result in the next proposition, we remark that $T$ is clearly multiply recurrent on levels by construction since in each column $C_n$, there are no spacers on the first $\lceil \sqrt{n} \rceil$ columns, and $\lceil \sqrt{n} \rceil \rightarrow \infty$ as $n \rightarrow \infty$. 

We say that a level $J$ is $(1-\delta)$-\textbf{full} of $A$ if $\mu(J \cap A) > (1-\delta) \cdot \mu(J)$. For any measurable $A$ of finite measure, we can find a level $J$ that is $(1-\delta)$-full for $A$ for any $\delta < 1$. Furthermore, if $J$ is a level in $C_n$ that is $(1-\delta)$-full of $A$, then for each $m \geq n$, $J$ has a descendant in $C_m$ that is also $(1 - \delta)$-full of $A$. In addition, we say that a level $J$ is $(1-\delta)$-\textbf{empty} of $A$ if $\mu(J \cap A) \leq (1-\delta) \cdot \mu(J)$. 

\begin{proposition} $T$ is multiply recurrent. 
\label{prop:mr}
\end{proposition}
\begin{proof} 
Let $A \subset X$ be of positive measure; we will show it is $k$-recurrent for all $k$. Fix $k \geq 1$. For $n$, we let $A_n$ be the subset of column $C_n$ in the first $\lceil \sqrt{n} \rceil$ subcolumns after cutting, and $B_n$ be the rest the column. 

For any $\delta$, we may find a level $J$ in some column $C_n$, $(1-\delta)$-full of $A$, for a $\delta$ we will choose in the future. We may assume that $n \geq k^2$, since if $n < k^2$, we could find a descendent of $J$ in $C_{k^2}$, still $(1-\delta)$-full of $A$, and use that decedent as $J$ instead. 

We claim that if $J \cap A_n$ is $\frac{2k-1}{2k}$ full of $A$, then $\mu(A \cap T^{h_n} A \cap \cdots \cap T^{(k-1)h_n}A) > 0$. To see this, we let $J_i$ be the segment of $J$ in the $i$th subcolumn and we break up the interval $J \cap A_n$ into segments, $\bigcup_{i=0}^{k}J_i, \bigcup_{i=k+1}^{2k} J_i, \ldots$, where the last segment will be $\bigcup_{i=sk+1}^{\lceil \sqrt{n}\rceil} J_i$ for the largest $s$ such that $sk+1 < \lceil \sqrt{n}\rceil $. 

If $J \cap A_n$ is $\frac{2k-1}{2k}$ full of $A$, then by a pigeonhole argument on these segments, at least one of the segments, $\bigcup_{i=0}^{k}J_i, \bigcup_{i=k+1}^{2k} J_i, \ldots \bigcup_{(s-1)k+1}^{sk} J_i$ will be $\frac{k-1}{k}$-full of $A$. We suppose without loss of generality that $I := \bigcup_{i=0}^{k}J_i$ is $\frac{k-1}{k}$-full of $A$. But then, $\mu((I \cap A) \cap T^{h_n} (I \cap A) \cap \cdots \cap T^{(k-1)h_n}(I \cap A)) > 0$ by a pigeonhole argument, so $A$ is $k$-recurrent. 

We suppose now that we chose our original $J$ such that $J$ is $\frac{4k-1}{4k}$-full of $A$ and $J$ is in $C_n$ for some $n \geq k^2$. We suppose for sake of contradiction that $J \cap A$ is not $k$-recurrent. Then, by our previous discussion, $J \cap A_m$ must be $\frac{2k-1}{2k}$-empty of $A$ for all $m \geq n$. Furthermore, we see that the proportion of $J$ not in $A_n \cup \ldots \cup A_{N}$ is 
$$B_{n, N} := \left(\frac{n-\lceil \sqrt{n}\rceil}{n}\right) \left( \frac{n+1 - \lceil \sqrt{n+1}\rceil}{n+1}\right)\cdots \left(\frac{N-\lceil \sqrt{N}\rceil}{N}\right) .$$ But we have that $\lim_{N \rightarrow \infty} B_{n,N}$ is 
\begin{align*} 
\prod_{k=n}^\infty \frac{n - \lceil \sqrt{n} \rceil}{n} \leq \prod_{k \geq n, k \text{ a square}} \frac{k - \lceil \sqrt{k}\rceil }{k} &=  \prod_{m=\lceil \sqrt{n} \rceil}^\infty \frac{m^2-m}{m^2}\\  &= \prod_{m=\lceil \sqrt{n} \rceil}^\infty \frac{m-1}{m},
\end{align*} 
which goes to $0$. Hence, for any $\varepsilon > 0$, we can find some $N$ such that $B_{n,N} < \varepsilon$. 

Then, since $A_n, \ldots, A_N$ are all $\frac{2k-1}{2k}$-empty of $A$, we have that $A_n \cup \ldots \cup A_N$ is $\frac{2k-1}{2k}$-empty of $A$. Since $\mu(J \cap A_n \cup \ldots \cup A_N) > (1-\varepsilon)\mu(J)$, and $J \backslash (A_n \cup \ldots \cup A_N)$ is $1$-empty of $A$ and has measure $< \varepsilon \mu(J)$, we have that $J$ is $\left(\frac{2k-1}{2k} (1-\varepsilon) + \varepsilon\right)$-empty of $A$. But then we may choose $\varepsilon$ small enough so $J$ $\frac{4k-1}{4k}$-empty of $A$, which is a contradiction. 

Hence, for some $m \geq n$, we must have $J \cap A_m$ is $\frac{2k-1}{2k}$ full of $A$, implying that $A$ is $k$-recurrent. 
\end{proof}

\subsection{Not Multiply Recurrent, Zero Type}
\label{sub:nmrzt}

A conservative ergodic Markov shift is zero type and can be chosen not to be $2$-recurrent \cite{AN00}.
For a rank-one example, the transformation $T$ in \cite[Section 3]{AdFrSi97} is such that $T\times T$ is not conservative, so it must be of zero type and it was shown in \cite{AtSi13} that is  not $2$-recurrent.

For completeness we give another  cutting and stacking transformation, $T$. We let $r_n = n+2$ for all $n$. Then, we add the appropriate number of spacers above each subcolumn to guarantee the following sets of heights. Again, we never add spacers on the first subcolumn. For each $n$, we want $$H_n = \{0\} \bigcup \left\{5^{\frac{n(n+1)}{2}}, 5^{\frac{n(n+1)}{2} + 1}, \ldots, 5^{\frac{n(n+1)}{2}+n}\right\}.$$

We notice that this transformation satisfies a steepness condition. That is, for any pair of distinct $h, h^\prime \in H$, the set of all heights, if $h^\prime > h$, we have that $h^\prime \geq 5 \cdot h$. Then, by Theorem 5.1 in \cite{DGPS}, we have that since $T$ is steep, normal, rank-one, and $\{r_n\}$ is nondecreasing with $\sup \{r_n\} = \infty$, then $T$ is zero-type. 

We can also see that $T$ is not multiply recurrent. In particular it is not $2$-recurrent. We consider the interval $I$. We suppose for sake of contradiction that $\mu(I \cap T^k I \cap T^{2k} I) > 0$ for some $k \geq 1$. Then, we may find some column $C_n$ such that $2k < \max\{D(I, n)\}$. We must then have that $$k, 2k \in S_n := \left\{\sum_{i=0}^n h_i : h_i \in H_i\right\}.$$ But we notice that all elements in $S_n$ are of the form $\sum_{i=0}^N \varepsilon_i \cdot 5^i$ where $N \in \Nm$ and $\varepsilon_i \in \{0,1\}$. But we can't have both $k$ and $2k$ of this form, as every natural number has a unique representation as a sum $\sum_{i=1}^N a_i \cdot 5^i$ where $a_i \in \{0,1,2,3,4\}$. Hence, we have a contradiction and $T$ could not have been $2$-recurrent. 

\subsection{Independence of Multiple Recurrence and the Zero Type Properties}

To complete the independence of multiple recurrence and the zero type,  we mention some examples. Eigen,  Hajian and Halverson in 1998 constructed examples that were multiply recurrent but not zero type (\cite{ESH98}). Furthermore, it is well known that the Hajian-Kakutani cutting and stacking transformation with $r_n=2$ and $2h_n$ spacers on the right subcolumn in each stage is neither multiply recurrent nor zero type. An number of examples that are multiply recurrent and not zero type, but weakly mixing for example, were constructed in \cite{DaSi04}. In \cite{DaRy10} Danilenko and Ryzhikov construct examples of ergodic multiply recurrent  transformations with arbitrary set of spectral multiplicities that are not of zero type by \cite{Ry11} since the weak closure of the powers is nontrivial.

\section{Rational Weak Mixing}
\label{sec:rwm}

In this section, we will give various conditions that imply that a transformation is  not rationally weakly mixing. We first introduce some notation. We fix a non-negative sequence $u = \{u_k\}$ such that $\sum_{k=0}^\infty u_k = \infty$ and let $a_u(n)$ denote $\sum_{k=0} ^{n-1}u_k$. Then, for a subset $K$ of $\Nm$, we define $a_u(K,n) = \sum_{k\in K\cap[0,n-1]}u_k$.

Now, we define some notions of smallness and density. If $\kappa \subset \Nm$, we say that a set $K$ is a $(u, \kappa)$-\textbf{small set} if 
$$\frac{a_u(K,n)}{a_u(n)}\underset{n \in \kappa}{\longrightarrow}0$$
as $n\to \infty$. Furthermore, we say a sequence $x_k\to L$ in $(u,\kappa)$-\textbf{density} to $L \in \mathbb{R}$ if there exists a $(u,\kappa)$-small set $K$ such that 
$x_k\to L$ as $k\to \infty$ for $k \in \kappa \backslash K$. We denote this type of convergence as $x_k\underset{(u,\kappa)-d}{\longrightarrow}L$. 
We say that $x_k$ converges $(u,\kappa)$-\textbf{strongly Cesaro} to $L \in \mathbb{R}$ if $$\frac{1}{a_u(n)}\sum_{k=0}^{n-1}u_k|x_k-L| \underset{k\rightarrow \infty,k \in \kappa}{\longrightarrow} 0. $$
We denote this type of convergence as  $x_k\underset{(u,\kappa)-s.C.}{\longrightarrow}L.$

We also say that two sequences $u = \{u_k\}$ and $v=\{v_k\}$ are $\kappa$-\textbf{asymptotic} (denoted by $u\overset{\kappa}{\approx} v$) if $$\frac{1}{a_u(n)}\sum_{k=0}^{n-1}|u_k-v_k|\underset{k\rightarrow \infty,k \in \kappa}{\longrightarrow} 0. $$
Similarly, a sequence $u=\{u_k\}$ is called $\kappa$-\textbf{smooth} if $$\frac{1}{a_u(n)}\sum_{k=0}^{n-1}|u_{k+1}-u_k|\underset{k\rightarrow \infty,k \in \kappa}{\longrightarrow} 0. $$

We start with the following theorem that gives conditions which lead to a transformation being not rationally weakly mixing. 

\begin{theorem}
\label{thm:rwmcond}
If $T$ is a normal, rank-one transformation such that for all $N \in \Nm$, $$\sum_{i=0}^{N-1}(d_i-d'_i)=\sum_{i=0}^{N-1}(e_i-e'_i)$$ for $d_i, d_i^\prime, e_i, e_i^\prime \in H_i$ implies that $d_i - d_i^\prime = e_i - e_i^\prime$ for all $0 \leq i \leq N-1$, then $T$ is not rationally weakly mixing. 
\end{theorem}

\begin{proof}
We suppose that we have such a transformation $T$, but that it is rationally weakly mixing. Then, it is subsequence rationally weakly mixing along the sequence $\kappa = \{M_m+1\}$. 

It can be shown that the sets $F \subset X$ that we have rational weak mixing along $\kappa$ with respect to is exactly $R_{\kappa}(T)$, the sets that we have weak rational ergodicity along $\kappa$ with respect to (\cite{Aa12}). We proved earlier that $T$ is subsequence weakly ergodic along $\kappa$ with respect to $I$ the level in $C_0$ and therefore any $J \subseteq I$ of positive measure. Define $u=\{u_k=\mu(J \cap T^kJ) \}$. 

We know by Aaronson \cite{Aa12} that if 
\begin{equation*}
\frac{\sum_{k=0}^{n-1}|\mu(J\cap T^kJ)-\mu(J)^2u_k(I)|}{a_n(I)} \rightarrow 0
\end{equation*}
then
\begin{equation*}
\frac{\mu(J\cap T^kJ)}{\mu(I\cap T^kI)} \underset{(u,\kappa)-d}{\longrightarrow} \frac{\mu(J)^2}{\mu(I)^2}.
\end{equation*}
Now we pick $J$ a descendant in column $C_j$ for some $j > 0$. Then, since $T$ is normal, by Lemma \ref{lem:normal}, we have that for some $N$ large, 
$$\frac{\mu(J\cap T^kJ)}{\mu(I\cap T^kI)}=\frac{w_N|D(J,N)\cap (k+D(J,N))|}{w_N|D(I,N)\cap (k+D(I,N)|}.$$ But then if $\mu(J \cap T^k J) > 0$, we can write $k=\sum_{i=0}^{N-1}k_i$, where $k_i \in D(J,i)-D(J,i)$. Moreover, $k = \sum_{i=j}^{N-1} (d_i-d_i^\prime) = \sum_{i=0}^{N-1} (e_i - e_i^\prime)$. Since we have that $d_i - d_i^\prime = e_i - e_i^\prime$ for all $j \leq i \leq N-1$, we have that 
\begin{align*} 
\frac{\mu(J \cap T^k J)}{\mu(I \cap T^k I)} & = \frac{\prod_{i=j}^{N}|\{(d_i,d^\prime_i):d_i-d^\prime_i=k_i\}|}{\prod_{i=0}^{N}|\{(e_i,e^\prime_i):e_i-e^\prime_i=k_i\}|} \\
& = \frac{1}{\prod_{i=0}^{j-1} \{(e_i, e_i^\prime) : e_i - e_i^\prime = 0\}|} \\
& = \frac{1}{\prod_{i=0}^{j-1} |H_i|}.
\end{align*}
Also, for $k$ such that $\mu(J \cap T^k J) = 0$, we have that $\frac{\mu(J \cap T^k J)}{\mu(I \cap T^k I)} = 0$. Hence, since $\frac{\mu(J)^2}{\mu(I)^2} = \frac{1}{\prod_{i=0}^{j-1} |H_i|^2}$, we have that 
 \begin{equation*}
 \frac{\mu(J\cap T^kJ)}{\mu(I\cap T^kI)} \underset{(u,\kappa)-d}{\nrightarrow} \frac{\mu(J)^2}{\mu(I)^2}.
\end{equation*} 
Hence, $T$ cannot be rationally weakly mixing. 
\end{proof}

Now, we use this theorem to give a few examples and classes of transformations that cannot be rationally weakly mixing. The following was shown in \cite{DGPS} under the normality assumption.

\begin{corollary}
\label{cor:steep}
If the transformation $T$ is steep, then $T$ is not rationally weakly mixing. 
\end{corollary} 
\begin{proof}
If $T$ is steep with steepness rate equal to $5$, then $\sum_{i=0}^{N}(d_i-d'_i)=\sum_{i=0}^{N}(e_i-e'_i)$ with $d_i,d'_i,e_i,e'_i \in H_i$ implies that $d_i-d'_i=e_i-e'_i$. It follows from Theorem \ref{thm:rwmcond} that $T$ is not rationally weakly mixing. 
\end{proof}
 
\begin{corollary}
The multiply recurrent, zero type transformation constructed in Section \ref{sec:ztmr} is not rationally weakly mixing. 
\end{corollary}
\begin{proof}
As one of the steps in the proof of the zero type property was proving that we had uniqueness of representation as sums of differences, Theorem \ref{thm:rwmcond} gives us that this transformation is not rationally weakly mixing. Hence, there exists a rank-one transformation that is zero-type, multiple recurrent and not rationally weak mixing. 
\end{proof}

\begin{corollary}
The multiply recurrent, not zero type transformation constructed in Section \ref{sec:ztmr} is not rationally weakly mixing. 
\end{corollary}
\begin{proof}
This transformation is steep, and therefore not rationally weakly mixing by Corollary \ref{cor:steep} and Theorem \ref{thm:rwmcond}. Hence, there exist multiply recurrent, not zero type transformations that are not rationally weakly mixing. 
\end{proof}

\begin{theorem} 
\label{thm:boundedrwm}
Let $T$ be a rank-one transformation. If for all columns $C_j$ and $J$ a level of $C_j$ with $J \subset I$, we have that 
\begin{equation*}
\label{eqn:limit:2}
\frac{\mu(J\cap T^{k+1}J)}{\mu(J\cap T^kJ)}\underset{(u, \kappa)-d}{\nrightarrow} 1,
\end{equation*}
where $\kappa=\{M_m+1\}_{m \geq 1}$, then $T$ is not rationally weakly mixing.

Furthermore, if there exists a level of a column $C_j$, $J \subset I$, such that $1 \notin (D(J,N)-D(J,N))-(D(J,N)-D(J,N))$,  for all sufficiently large $N$ and $T$ is a rank-one transformation, then we have that $T$ is not rationally weakly mixing. 
\end{theorem}

\begin{proof}
Suppose, by contradiction, that $T$ is rationally weak mixing. By the same argument as in \ref{thm:rwmcond}, $T$ is subsequence rationally weakly mixing along $\kappa$ and $I=C_0 \in R_{\kappa}(T)$. If $J \subseteq I$, then $J \in R_{\kappa}(T)$ and $T(J) \in R_{\kappa}(T)$. This proves that $\mu(J\cap T^kJ) \overset{\kappa}{\approx} \mu(J\cap T^{k+1}J)$, meaning that $\mu(J \cap T^k J)$ is a $\kappa$-smooth sequence for levels $J \subseteq I$ of the columns $C_j$. Denote $v_k=\mu(J\cap T^kJ)$ and let $v=(v_0,v_1,\ldots).$ As $T$ is subsequence rationally weakly mixing along $\kappa$, then $u\overset{\kappa}{\approx} v$.

By Proposition 3.2 in \cite{Aa12}, this means that the sequence $$\frac{\mu(J\cap T^{k+1}J)}{\mu(J \cap T^k J)} \underset{(v, \kappa)-s.C.}{\longrightarrow} 1,$$ 
which implies by Remark 3.3 iii) of the same paper that $$\frac{\mu(J\cap T^{k+1}J)}{\mu(J\cap T^kJ)}\underset{(v, \kappa)-d}{\longrightarrow} 1.$$
 By Remark 3.3 ii) in \cite{Aa12}, it follows that $$\frac{\mu(J\cap T^{k+1}J)}{\mu(J\cap T^kJ)}\underset{(u, \kappa)-d}{\longrightarrow} 1,$$ as $u\overset{\kappa}{\approx} v$, and we have a contradiction.

Now we prove the second half of the theorem. By the first part, it suffices to show for $J$ that $$ \frac{\mu(J\cap T^{k+1}J)}{\mu(J\cap T^kJ)}\underset{(u,\kappa)-d}{\nrightarrow} 1. $$
But as $1 \notin (D(J,N)-D(J,N))-(D(J,N)-D(J,N))$ we get that for any fixed $k$, $\mu(J\cap T^{k+1}J)$ and $\mu(J\cap T^kJ)$ cannot be both be greater than $0$ at the same time. Hence, the limit cannot go to $1$ and we are done. \end{proof}

\begin{remark}
We see that the above proof shows even more than the conclusion of theorem, it shows that $T$ is not subsequence rationally weakly mixing along the sequence $\kappa=\{M_m+1\}$. The importance of this subsequence comes from the fact that all rank-one transformations are subsequence weakly rationally ergodic along $\kappa$. 
\end{remark}

\appendix
\section{}

\subsection{An Equivalent Definition of Rank-one Transformations}
The definition that we gave of a rank-one transformation was one for which there exists a sequence of Rokhlin columns $\{C_n\}$ such that for any $A \subset X, \mu(A) < \infty$, and $\varepsilon > 0$, there exists an $N$ such that for all $n \geq N$, $C_n$ approximates $A$ within $\varepsilon$. If a transformation satisfies this property, then we say that $C_n \rightarrow \varepsilon$, or that the columns \textbf{converge to the point partition} on $X$. 

A Rokhlin column $C_{n+1}$ is said to be a \textbf{refinement} of $C_n$ if for each level $A_i$ in $C_n$, $A_i$ can be written as a union of levels $B_j$ of $C_{n+1}$. As we often use in our paper that our chosen Rokhlin columns for our transformations are refining, we prove that this is a valid assumption here. 

\begin{proposition} 
\label{prop:refining}
Every rank-one transformation has a sequence of Rokhlin columns converging to the point partition such that columns are refining. 
\end{proposition}

The proof of this in the finite case can be found in \cite{B71}, Lemma 9. The proof of this in the infinite case is similar, but we reproduce it here for completeness. We first introduce some more notation. If $C = \{A_i : 1 \leq i \leq m\}$ and $D = \{B_j : 1 \leq j \leq n\}$, we write that $C \leq D$ if each $A_i$ can be written as a union of terms in $D$. If $E$ is a set, then we write $E \leq D$ to mean $\{E\} \leq D$. 

If $m = n$, then we let $\rho (C, D) = \sum_{i=1}^m \mu(A_i \triangle B_i)$ and for any $n$ and $m$, we let $$P(C,D) = \min\{\rho(C, D^\prime) : D^\prime \leq D\},$$ where the $D^\prime$s must be partitions into $m$ pieces. Since we are dealing with rank-one transformations, we often want to deal with Rokhlin columns as partitions. Hence, we define also $$P_T(C,D) = \min\{\rho(C, D^\prime) : D^\prime \leq D, D^\prime \text{ is a Rokhlin column for } T\}.$$
As an immediate consequence of these definitions, we have the following lemma: 

\begin{lemma} 
\label{lem:partbound}
Let $C$ and $C^\prime$ be Rokhlin columns for $T$ of the same height and $D$ be a partition. Then, $$P_T(D, C) \leq P_T(D, C^\prime) + \rho (D^\prime, D).$$
\end{lemma}
We also have the following lemma. 

\begin{lemma} 
\label{lem:columnbound}
Let $C = \{A_i : 1 \leq i \leq m\}$ and $D = \{B_j : 1 \leq j \leq n\}$ be Rokhlin columns for $T$. Then, $$P_T(D, C) \leq nP(B_1, C) + n \mu(A_1).$$
\end{lemma}

\begin{proof} We first let $F$ be a union of sets in $C$ such that $\mu(B_1 \triangle F) = P(B_1, C)$. Then, we let $G$ be the set obtained from $F$ by first removing all $A_i$ such that $\mu(B_1 \cap A_i) \leq \frac12 \mu(A_i)$ and then removing the $A_i$ of largest index still in $F$. Now if we let $\mathcal{I}$ be the set of indices of $A_i$ terms still in $F$, we have from construction that any two indices in $\mathcal{I}$ must be at least $n$ and that the largest index is less than $m-n+1$. 

Now, we let $B_j^\prime = T^{j-1}G$ and define a Rokhlin column $D^\prime = \{B_j^\prime : 1 \leq j \leq n\}$. Since $G$ is a union of $A_i$ terms for $1 \leq i \leq m-n+1$, we have that each $T^{j-1}G$ is also a union of terms in $C$. Thus, $D^\prime \leq C$. Thus, we have that $P_T(D,C) \leq \rho(D, D^\prime)$. 

By construction, we have that $\mu(B_1 \triangle G) \leq D(B_1, C) + \mu(A_1)$. The lemma then follows as $\rho(D, D^\prime) = n\mu(B_1 \triangle G)$ as a consequence of $T$ being measure-preserving. 
\end{proof}

Finally, we can prove that every rank-one transformation has a sequence of Rokhlin columns converging to the point partition that are refining, as stated in Proposition \ref{prop:refining}.

\begin{proof} [Proof of Proposition \ref{prop:refining}]
By Lemma \ref{lem:columnbound}, if $C_k$ are a sequence of Rokhlin columns for $T$, we may assume that $P_T(C_k, C_{k+1}) < \delta_k$ where $\sum_{k=1}^\infty \delta_k < \infty$. If the columns did not already satisfy this condition, we could choose a subset of them that did by the bound in the lemma. 

We now define a doubly infinite sequence of Rokhlin columns, from which we will obtain our desired sequence of refining columns. We begin by letting $C_k^0 = C_k$ for all $k \geq 1$. Now, assuming that $C_k^r$ has already been defined for all $k \in \Nm$, we let $C_k^{r+1}$ be some Rokhlin column such that $$\rho(C_k^r, C_k^{r+1}) = P_T(C_k^r, C_{k+1}^r).$$ Furthermore, we can choose $C_{k}^{r+1}$ to be exactly the column $D \leq C_{k+1}^r$ such that $\rho(D, C_{k+1}^r)$ is minimized, giving us that $$P_T(C_k^{r+1}, C_{k+1}^r) = 0.$$

We notice that for $r=0$ and for all $k \geq 1$, we have that 
\begin{equation}
\label{eqn:inductbound}
\rho(C_k^r, C_k^{r+1}) < \delta_{r+k}.
\end{equation}
We claim that \eqref{eqn:inductbound} in fact holds for all $r \geq 0$ and $k \geq 1$. We show this by induction. The base case $r=0$ holds by earlier assumptions. If we have now that \eqref{eqn:inductbound} holds for a fixed $r$ and all $k \geq 1$, by Lemma \ref{lem:partbound}, we have that 
\begin{align*} 
\rho(C_k^{r+1}, C_k^{r+2}) & = P_T(C_k^{r+1}, C_{k+1}^{r+1}) \\
& \leq P_T(C_k^{r+1}, C_{k+1}^{r}) + \rho(C_{k+1}^{r}, C_{k+1}^{r+1}) \\
& \leq \delta_{r+1+k}
\end{align*}
for any $k \geq 1$. Hence, we have that \eqref{eqn:inductbound} holds for all $r \geq 0$ and $k \geq 1$. But now, we have for all $r \geq 0, s \geq 0, k \geq 1$ that
$$\rho(C_k^r, C_k^{r+s}) < \sum_{i=0}^{s-1} \delta_{i+r+k}.$$ Hence, since we had that $\sum_{k=1}^\infty \delta_k < \infty$, for a fixed $k$ the columns $C_k^r$ form a Cauchy sequence. Thus, for each $k$, we can find a Rokhlin column $D_k$ such that $\rho(C_k^r, D_k) \rightarrow 0$ as $r \rightarrow \infty$. Since $P_T(C_k^{r+1}, C_{k+1}^r) = 0$, we can take a limit as $r \rightarrow \infty$ to see that $D_k \leq D_{k+1}$ for each $k \geq 1$. 

Finally, since $C_k$ converges to the point partition and $\rho(C_k, D_k) \leq \sum_{i=k}^\infty \delta_i$, which goes to zero as $k \rightarrow \infty$, we have that $D_k$ also converges to the point partition. Hence, taking $D_k$ to be our columns, we have a sequence of Rokhlin columns for $T$ that are refining. 
\end{proof}

\bibliographystyle{plain}
\bibliography{ReferencesErgcopy2}

\end{document}